\documentclass[final,leqno,11pt]{siamltex}
\usepackage[left=1.5in,top=1.5in,right=1.5in,bottom=1.3in,letterpaper]{geometry}
\usepackage{setspace}
\usepackage{amsmath}
\usepackage{amssymb}
\usepackage{amsfonts}
\usepackage{graphicx}
\usepackage{booktabs}
\graphicspath{{images/}}
\usepackage{epsfig}
\usepackage{color}
\usepackage[color,final]{showkeys}
\usepackage[ruled,vlined]{algorithm2e}
\usepackage{subfigure}

\newtheorem{remark}{Remark}[section]

\newcommand{\bfx}{{\bf x}}
\newcommand{\bfy}{{\bf y}}
\newcommand{\bfz}{{\bf z}}
\newcommand{\bfu}{{\bf u}}

\newcommand{\bfa}{{\bf a}}
\newcommand{\bfb}{{\bf b}}

\newcommand{\va}{{\mathbf{a}}}
\newcommand{\vb}{{\mathbf{b}}}

\newcommand{\vd}{{\mathbf{d}}}

\newcommand{\vr}{{\mathbf{r}}}

\newcommand{\vw}{{\mathbf{w}}}

\newcommand{\vI}{{\mathbf{I}}}

\newcommand{\cL}{{\mathcal{L}}}

\newcommand{\cN}{{\mathcal{N}}}

\newcommand{\cX}{{\mathcal{X}}}

\newcommand{\cZ}{{\mathcal{Z}}}

\newcommand{\Dx}{\Delta \bfx}

\newcommand{\Dlambda}{\Delta \lambda}

\newcommand{\RR}{\mathbb{R}}

\newcommand{\st}{{\text{s.t.}}} 
\newcommand{\dom}{{\mathrm{dom}}} 

\DeclareMathOperator*{\argmin}{arg\,min}

\newcommand{\bc}{\begin{center}}
\newcommand{\ec}{\end{center}}

\newcommand{\bdm}{\begin{displaymath}}
\newcommand{\edm}{\end{displaymath}}

\newcommand{\beq}{\begin{equation}}
\newcommand{\eeq}{\end{equation}}

\newcommand{\bfl}{\begin{flushleft}}
\newcommand{\efl}{\end{flushleft}}

\newcommand{\bt}{\begin{tabbing}}
\newcommand{\et}{\end{tabbing}}

\newcommand{\beqn}{\begin{align}}
\newcommand{\eeqn}{\end{align}}

\newcommand{\beqs}{\begin{align*}} 
\newcommand{\eeqs}{\end{align*}}  

\newtheorem{assumption}{Assumption}

\title{Parallel Multi-Block ADMM with \MakeLowercase{$o(1/k)$} Convergence}

\author{Wei Deng\thanks{Department of Computational and Applied Mathematics,
Rice University, Houston, TX 77005. ({\tt wei.deng@rice.edu})}
\and Ming-Jun Lai\thanks{Department of Mathematics, University of Georgia, Athens, GA 30602.
({\tt mjlai@math.uga.edu})}
\and Zhimin Peng$^\ddag$
\and Wotao Yin\thanks{Department of Mathematics, University of California,
Los Angeles, CA 90095. ({\tt zhimin.peng / wotaoyin@math.ucla.edu})}}

\begin{document}

\maketitle

\begin{abstract}
This paper introduces a parallel and distributed extension to the alternating direction method of multipliers (ADMM) for solving convex problem:
\begin{align*}
\textrm{minimize} ~~& f_1(\bfx_1) + \cdots + f_N(\bfx_N)\\
\textrm{subject to}~~ & A_1 \bfx_1 ~+ \cdots +
A_N\bfx_N =c,\\
& \bfx_1\in \cX_1,~\ldots, ~\bfx_N\in \cX_N.
\end{align*}
The algorithm decomposes the original problem into $N$ smaller subproblems and solves them in parallel at each iteration. This Jacobian-type algorithm is well suited for distributed computing and  is particularly attractive for solving certain large-scale problems. 

This paper introduces a few novel results. Firstly, it shows that extending  ADMM straightforwardly from the classic Gauss-Seidel setting to the  Jacobian setting,  from 2 blocks to  $N$ blocks, will preserve convergence if matrices $A_i$ are mutually near-orthogonal and have full column-rank. Secondly, for general matrices $A_i$, this paper proposes to add proximal terms of different kinds  to the $N$ subproblems so that  the subproblems can be solved in flexible and efficient ways and the algorithm converges globally at a rate of $o(1/k)$. Thirdly, a simple technique is introduced to improve some existing convergence rates from $O(1/k)$ to $o(1/k)$.

In practice, some conditions in our convergence theorems are conservative. Therefore, we introduce a  strategy for dynamically tuning the parameters in the algorithm, leading to substantial acceleration of the convergence in practice. Numerical results are presented to demonstrate the efficiency of the proposed method in comparison with several existing parallel algorithms.

We implemented our algorithm on  Amazon EC2, an on-demand public computing cloud, and report its performance on very large-scale basis pursuit problems with distributed data.
\end{abstract}


\pagestyle{myheadings}
\thispagestyle{plain}
\markboth{Wei Deng, Ming-Jun Lai, Zhimin Peng, Wotao Yin}{Parallel multi-block ADMM with $o(1/k)$ convergence}

\section{Introduction}
We consider the following convex optimization problem with $N~(N\geq 2)$ blocks of variables:
\begin{equation}\label{multivariateOptimization}
\min_{\bfx_1,\bfx_2,\ldots,\bfx_N} ~\sum_{i=1}^N f_i(\bfx_i) \quad\st ~\sum_{i=1}^N
A_i\bfx_i =c,
\end{equation}
where $\bfx_i\in \RR^{n_i}$, $A_i\in\RR^{m\times n_i}$, $c\in\RR^m$, and $f_i:\RR^{n_i}\rightarrow (-\infty,+\infty]$ are closed proper convex functions, $i=1,2,\ldots,N$.
If an individual block is subject to constraint $\bfx_i\in \cX_i$, where $\cX_i\subseteq \RR^{n_i}$ is a nonempty closed convex set, it can be
incorporated in the objective function $f_i$ using the indicator function:
\begin{equation} \label{indicator}
I_{\cX_i}(\bfx_i)=\left\{
\begin{array}{cl}
0 &\mbox{if}~\bfx_i \in \cX_i,\\
+\infty &\mbox{otherwise.}
\end{array}
\right.
\end{equation}
The problem \eqref{multivariateOptimization} is also referred to as an \textit{extended monotropic programming problem} \cite{bertsekas2008extended}. The special case that each $\bfx_i$ is a scalar (i.e., $n_i=1$) is called a \textit{monotropic programming problem} \cite{rockafellar1981monotropic}.
Such optimization problems arise from a broad spectrum of applications including numerical partial differential
equations, signal and image processing, compressive sensing, statistics and machine learning. See \cite{glowinski1984numerical,Bertsekas-Tsitsiklis-book-97,boyd2010distributed,tao2011recovering,chandrasekaran2012latent,peng2012rasl,mota2013d,wei20131,parikh2013block} and the references therein for a number of examples.

In this paper, we focus on \textit{parallel and distributed} optimization algorithms for solving the problem \eqref{multivariateOptimization}. Since both of the objective function and constraints of \eqref{multivariateOptimization} are summations of terms on individual $\bfx_i$'s (we call them \textit{separable}), the problem can be decomposed into $N$ smaller subproblems, which can be solved in a parallel and distributed manner.

\subsection{Literature review}
A simple distributed algorithm for solving \eqref{multivariateOptimization} is \textit{dual decomposition} \cite{everett1963generalized}, which is essentially a \textit{dual ascent method} or \textit{dual subgradient method} \cite{shor1985minimization} as follows. Consider the Lagrangian for problem \eqref{multivariateOptimization}:
\beq \label{Lagrange}
\cL(\bfx_1,\ldots,\bfx_N,\lambda)=\sum_{i=1}^N f_i(\bfx_i)-\lambda^\top
\left(\sum_{i=1}^N A_i\bfx_i -c\right)
\eeq
where $\lambda\in \RR^m$ is the Lagrangian multiplier or the dual variable.
The method of dual decomposition iterates as follows: for $k\geq 1$,
\beq
\left\{
\begin{array}{l}
(\bfx_1^{k+1},\bfx_2^{k+1},\ldots,\bfx_N^{k+1}) =\argmin_{\{\bfx_i\}}~
\cL(\bfx_1,\ldots,\bfx_N,\lambda^k),\medskip\\
\lambda^{k+1}  = \lambda^k - \alpha_{k}\left(\sum_{i=1}^N A_i\bfx_i^{k+1} -c\right),
\end{array}\right.
\eeq
where $\alpha_{k} > 0$ is a step-size.
Since all the $\bfx_i$'s are separable in the Lagrangian function \eqref{Lagrange}, the $\bfx$-update step reduces to solving $N$ individual $\bfx_i$-subproblems:
\beq
\bfx_i^{k+1} = \argmin_{\bfx_i} f_i(\bfx_i)-\langle\lambda^{k}, A_i \bfx_i\rangle, \mbox{ for } i=1,2,\ldots,N,
\eeq
and thus they can be carried out in parallel.
With suitable choice of $\alpha_k$ and certain assumptions, dual decomposition is guaranteed to converge to an
optimal solution \cite{shor1985minimization}. However, the convergence of such subgradient method often tends to be slow in practice. Its convergence rate for general convex problems is $O(1/\sqrt{k})$.

Another effective distributed approach is based on the \textit{alternating direction method of multipliers} (ADMM). ADMM was introduced in \cite{gabay1976dual,glowinski1975approximation} to solve the special case of problem \eqref{multivariateOptimization}
with two blocks of variables ($N=2$). It utilizes the \textit{augmented Lagrangian} for \eqref{multivariateOptimization}:
\beq \label{AL}
\cL_{\rho}(\bfx_1,\ldots,\bfx_N,\lambda)=\sum_{i=1}^N f_i(\bfx_i)-\lambda^\top
\left(\sum_{i=1}^N A_i\bfx_i -c
\right)+\frac{\rho}{2}\left\|\sum_{i=1}^N A_i\bfx_i -c\right\|_2^2,
\eeq
which incorporates a quadratic penalty of the constraints (with a parameter $\rho>0$) into the Lagrangian.
In each iteration, the augmented Lagrangian is minimized over $\bfx_1$ and $\bfx_2$ separately, one after the other, followed by a dual update for $\lambda$. The iterative scheme of ADMM is outlined below:
\beq
\left\{
\begin{array}{l}
\bfx_1^{k+1}=\argmin_{\bfx_1}\cL_{\rho}(\bfx_1,\bfx_2^k,\lambda^k),\medskip\\
\bfx_2^{k+1}=\argmin_{\bfx_2}\cL_{\rho}(\bfx_1^{k+1},\bfx_2,\lambda^k),\medskip\\
\lambda^{k+1}=\lambda^k-\rho(A_1\bfx_1^{k+1}+A_2\bfx_2^{k+1}-c).
\end{array}\right.
\eeq
To solve the problem \eqref{multivariateOptimization} with $N\geq 3$ using ADMM, one can first convert the multi-block problem into an equivalent two-block problem via variable splitting \cite{Bertsekas-Tsitsiklis-book-97,boyd2010distributed,wang2013solving}:
\begin{equation} \label{split}
\begin{split}
\min_{\{\bfx_i\},\{\bfz_i\}} &\quad\sum_{i=1}^N f_i(\bfx_i)+I_{\cZ}(\bfz_1,\ldots,\bfz_N)\\
\st &\quad A_i\bfx_i-\bfz_i=\frac{c}{N},~\forall i=1,2,\ldots,N,
\end{split}
\end{equation}
where $I_{\cZ}$ is a indicator function defined by \eqref{indicator}, and the convex set $\cZ$ is given by
\[\cZ=\left\{(\bfz_1,\ldots,\bfz_N): \sum_{i=1}^N\bfz_i=0\right\}.\]
The variables in \eqref{split} can be grouped into two blocks: $\bfx:=(\bfx_1,\ldots,\bfx_N)$ and $\bfz:=(\bfz_1,\ldots,\bfz_N)$, so that ADMM can directly apply. The augmented Lagrangian for \eqref{split} is given by
\beq
\cL_{\rho}(\bfx,\bfz,\lambda)=\sum_{i=1}^N f_i(\bfx_i)+I_{\cZ}(\bfz)-\sum_{i=1}^N
\lambda_i^\top\left(A_i\bfx_i-\bfz_i-\frac{c}{N}\right)+\frac{\rho}{2}\sum_{i=1}^N\left\|A_i\bfx_i-\bfz_i-\frac
{c}{N}\right\|_2^2.
\eeq
Since all the $\bfx_i$'s are now fully decoupled, the resulting $\bfx$-subproblem decomposes into $N$ individual $\bfx_i$-subproblems, which can be carried out in parallel.
The resulting $\bfz$-subproblem is a simple quadratic problem:
\beq
\bfz^{k+1}=\argmin_{\{\bfz:\sum_{i=1}^N\bfz_i=0\}}~\sum_{i=1}^N\frac{\rho}{2}\left\|A_i\bfx_i-\bfz_i-\frac
{c}{N}-\frac{\lambda_i^k}{\rho}\right\|_2^2,
\eeq
which admits a closed-form solution. The details of the algorithm are summarized below:

\begin{algorithm}[H]
Initialize $\bfx^0,~\lambda^0,~\rho>0$\;
\For{$k=0,~1,\ldots$}{
Update $\bfz_i$ then $\bfx_i$ for $i=1,\ldots,N$ \emph{in parallel} by: 
$\bfz^{k+1}_i=\left(A_i\bfx_i^k-\frac{c}{N}-\frac{\lambda^k_i}{\rho}\right)-\frac{1}{N}\left\{\sum_{j=1}^N
A_j\bfx_j^k-\frac{c}{N}-\frac{\lambda^k_j}{\rho}\right\}$\;
$\bfx_i^{k+1}=\argmin_{\bfx_i} f_i(\bfx_i)+\frac{\rho}{2}\left\|A_i\bfx_i-\bfz_i^{k+1}-\frac{c}{N}-\frac{\lambda^k_i}{\rho}\right\|_2^2$\;
Update $\lambda_i^{k+1}=\lambda_i^{k}-\rho\left(A_i\bfx^{k+1}_i-\bfz^{k+1}_i-\frac{c}{N}\right),\forall i=1,\ldots,N$.}
\caption{Variable Splitting ADMM (VSADMM)} \label{alg:vsadmm}
\end{algorithm}
The distributed ADMM approach based on \eqref{split}, by introducing splitting variables, \emph{substantially increases the number of variables and constraints} in the problem, especially when $N$ is large.

We prefer to extending the ADMM framework for solving \eqref{multivariateOptimization} rather than first converting \eqref{multivariateOptimization} to a two-block problem and then applying the classic ADMM. A natural extension is to simply replace the two-block alternating minimization scheme by a sweep of Gauss-Seidel update, namely, update $\bfx_i$  for  $i=1,2,\ldots,N$ sequentially as follows:
\begin{align} \label{G-S update}
\nonumber
\bfx_i^{k+1}&=\argmin_{\bfx_i}~\cL_{\rho}(\bfx_1^{k+1},\ldots,\bfx_{i-1}^{k+1},\bfx_i,\bfx_{i+1}^k,\ldots,\bfx_N^{k},\lambda^k)\\
&=\argmin_{\bfx_i}~f_i(\bfx_i)+\frac{\rho}{2}\left\|\sum_{j<i} A_j\bfx_j^{k+1}+A_i\bfx_i+\sum_{j>i}A_j\bfx_j^k-c-\frac{\lambda^k}{\rho}\right\|_2^2.
\end{align}
Such \textit{Gauss-Seidel ADMM} (Algorithm \ref{alg:admm}) has been considered lately, e.g., in \cite{he2012alternating,luo2012linear}.
However, it has been shown that the algorithm may not  converge for $N\geq 3$ \cite{chen2013direct}. Although lack of convergence guarantee, some empirical studies show that Algorithm \ref{alg:admm} is still very effective at solving many practical problems
(see, e.g., \cite{peng2012rasl,tao2011recovering,wang2013solving}).

\begin{algorithm}[H]
Initialize $\bfx^0,~\lambda^0,~\rho>0$\;
\For{$k=0,~1,\ldots$}{
Update $\bfx_i$ for $i=1,\ldots,N$ \emph{sequentially} by:
$\bfx_i^{k+1}=\min_{\bfx_i} f_i(\bfx_i)+\frac{\rho}{2}\left\|\sum_{j<i}
A_j\bfx_j^{k+1}+A_i\bfx_i+\sum_{j>i}A_j\bfx_j^k-c-\frac{\lambda^k}{\rho}\right\|_2^2$;
Update $\lambda^{k+1}=\lambda^{k}-\rho\left(\sum_{i=1}^N A_i\bfx_i^{k+1} -c\right)$.}
\caption{Gauss-Seidel ADMM} \label{alg:admm}
\end{algorithm}
A disadvantage of Gauss-Seidel ADMM is that the blocks are updated one after another, which is not amenable for parallelization.

\subsection{Jacobian scheme}
To overcome this disadvantage, this paper considers using a Jacobi-type scheme that updates all the $N$ blocks in parallel:
\begin{align} \label{Jacobian update}
\nonumber
\bfx_i^{k+1}&=\argmin_{\bfx_i}~\cL_{\rho}(\bfx_1^{k},\ldots,\bfx_{i-1}^{k},\bfx_i,\bfx_{i+1}^k,\ldots,\bfx_N^
{k},\lambda^k)\\
&=\argmin_{\bfx_i}~f_i(\bfx_i)+\frac{\rho}{2}\left\|A_i\bfx_i+\sum_{j\ne i}A_j\bfx_j^k-c-\frac{\lambda^k}{\rho}
\right\|_2^2,~\forall i=1,\ldots,N.
\end{align}
We refer to it as \textit{Jacobian ADMM}; see Algorithm \ref{alg:Jacobian-admm}.

\begin{algorithm}[H]
Initialize $\bfx^0,~\lambda^0,~\rho>0$\;
\For{$k=0,~1,\ldots$}{
Update $\bfx_i$ for $i=1,\ldots,N$ \emph{in parallel} by:
$\bfx_i^{k+1}=\argmin_{\bfx_i}f_i(\bfx_i)+\frac{\rho}{2}\left\|A_i\bfx_i+\sum_{j\ne i}A_j\bfx_j^k-c-\frac{\lambda^k}{\rho}\right\|_2^2$\;
Update $\lambda^{k+1}=\lambda^{k}-\rho\left(\sum_{i=1}^N A_i\bfx_i^{k+1} -c\right)$.}
\caption{Jacobian ADMM} \label{alg:Jacobian-admm}
\end{algorithm}
The parallelization comes with a cost: this scheme is more likely to diverge than the Gauss-Seidel scheme for the same parameter $\rho$. In fact, it may diverge even in the two-block case; see \cite{he2013full} for such an example. In order to guarantee its convergence, either additional assumptions or modifications to the algorithm must be made.

In Section \ref{sec:suf}, we show that if matrices $A_i$ are mutually near-orthogonal and have full column-rank, then Algorithm 3 converges globally. For general cases, a few variants of Jacobian ADMM have been proposed in \cite{he2009parallel,he2013full} by taking additional correction steps at every iteration.

In this paper, we propose \textit{Proximal Jacobian ADMM}; see Algorithm \ref{Parallel_N}). Compared with Algorithm \ref{alg:Jacobian-admm}, there are a proximal term $\frac{1}{2}\|\bfx_i-\bfx_i^k\|_{P_i}^2$ for each $\bfx_i$-subproblem and  a damping parameter $\gamma>0$ for the update of $\lambda$. Here $P_i\succeq 0$ is some symmetric and positive semi-definite matrix and we let $\|\bfx_i\|^2_{P_i}:= \bfx_i^\top P_i\bfx_i$.

\begin{algorithm}[H]
Initialize: $\bfx^0_i~(i=1,2,\ldots, N)$ and $\lambda^0$\;
\For{$k=0,1,\ldots$}{
{Update $\bfx_i$ for $i=1,\ldots,N$ \emph{in parallel} by:}
$\bfx_i^{k+1} =\argmin_{\bfx_i} f_i(\bfx_i)+\frac{\rho}{2}\left\|A_i\bfx_i+\sum_{j\ne
i}A_j\bfx_j^k-c-\frac{\lambda^k}{\rho}\right\|_2^2+\frac{1}{2}\left\|\bfx_i-\bfx_i^k\right\|
_{P_i}^2$\;

Update $\lambda^{k+1}=\lambda^k -\gamma\rho(\sum_{i=1}^N A_i \bfx_i^{k+1}- c)$.
}
\caption{Proximal Jacobian ADMM}
\label{Parallel_N}
\end{algorithm}
The proposed algorithm has a few advantages.
First of all, as we will show, it enjoys global convergence as well as an $o(1/k)$ convergence rate under conditions on  $P_i$ and $\gamma$.
Secondly, when the $\bfx_i$-subproblem is not strictly convex, adding the proximal term can make the subproblem strictly or strongly convex, making it more stable.
Thirdly, we provide multiple choices for matrices $P_i$ with which the subproblems  can be made easier to solve.
Specifically, the $\bfx_i$-subproblem contains a quadratic term $\frac{\rho}{2}\bfx_i A_i^\top A_i\bfx_i$. When   $A_i^\top A_i$ is ill-conditioned or computationally expensive to invert, one can let $P_i = D_i-\rho A_i^\top A_i$, which cancels the quadratic term $\frac{\rho}{2}\bfx_i A_i^\top A_i\bfx_i$ and adds $\frac{1}{2}\bfx_i D_i\bfx_i$. The matrix $D_i$ can be chosen as some well-conditioned and simple matrix (e.g., a diagonal matrix), thereby leading to an easier subproblem.

Here we mention two commonly used choices of $P_i$:
\begin{itemize}
\item $P_i=\tau_i \vI~(\tau_i>0)$: This corresponds to the standard \textit{proximal method}.
\item $P_i=\tau_i \vI-\rho A_i^\top A_i~(\tau_i>0)$: This corresponds to the
\textit{prox-linear method} \cite{chen1994proximal}, which linearizes the quadratic penalty term of augmented Lagrangian at the current point $\bfx_i^k$ and adds a proximal term.
More specifically, the prox-linear $\bfx_i$-subproblem is given by
\beq
\bfx_i^{k+1}=\argmin_{\bfx_i}~f_i(\bfx_i)+\left\langle \rho A_i^\top(A\bfx^k-c-
\lambda^k/\rho),\bfx_i\right\rangle+\frac{\tau_i}{2}
\left\|\bfx_i-\bfx_i^k\right\|^2.
\eeq
It essentially uses an identity matrix $\tau_i\vI$ to approximate the Hessian matrix $\rho A_i^\top A_i$ of the quadratic penalty term.
\end{itemize}
More choices of $P_i$ have also been discussed in \cite{zhang2011unified,deng2012linear}.

\subsection{Summary of Contributions}
This paper introduces a few novel results from different perspectives.
Firstly, we propose \textit{Proximal Jacobian ADMM}, which is suitable for parallel and distributed computing.
The use of flexible proximal terms make it possible to solve its subproblems in different ways, important for easy coding and fast computation. We establish its convergence at a rate of $o(1/k)$.
Our numerical results on the exchange problem and $\ell_1$-minimization problem show that the proposed algorithm achieves competitive performance, in comparison with several existing parallel algorithms.

The second contribution is a sufficient condition of convergence for the extension of classic  ADMM with a Jacobian update scheme.

The third contribution is the improvement of the established convergence rate of $O(1/k)$ for the standard ADMM to $o(1/k)$ by a simple proof.
This technique can be also applied to various other algorithms, and some existing convergence rates of $O(1/k)$ can be slightly improved to $o(1/k)$ as well.

\subsection{Preliminary, Notation, and Assumptions}
To simplify the notation in this paper, we introduce
\[
\bfx:=\begin{pmatrix}
\bfx_1\\
\vdots\\
\bfx_N
\end{pmatrix}\in\RR^{n},
~A:=\begin{pmatrix}
A_1,\ldots,A_N
\end{pmatrix}\in \RR^{m\times n},\\
~\bfu:=\begin{pmatrix}
\bfx\\
\lambda
\end{pmatrix}\in\RR^{n+m},\]
where $$n=\sum_{i=1}^N n_i.$$
We let $\langle\cdot,\cdot\rangle$ and $\|\cdot\|$ denote the standard inner product and $\ell_2$-norm $\|\cdot\|_2$, respectively, in the Euclidean space. For a matrix $M\in \RR^{l\times l}$, $\|M\|$ denotes the spectral norm, i.e., the largest singular value of $M$.
For a positive definite matrix $G\in \RR^{l\times l}$, we define the \textit{$G$-norm} as follows:
\beq
\|\bfz\|_G:=\sqrt{\bfz^\top G \bfz},~\forall \bfz\in \RR^l.
\eeq
If the matrix $G$ is positive semi-definite, then $\|\cdot\|_G$ is a semi-norm.

Throughout the paper, we make the following standard assumptions.
\begin{assumption} \label{ass:convex}
Functions $f_i:\RR^{n_i}\rightarrow (-\infty,+\infty]~(i=1,2,\ldots,N)$ are closed proper convex.
\end{assumption}

\begin{assumption}
There exists a saddle point $\bfu^*=(\bfx^*_1,\bfx^*_2,\ldots,\bfx^*_N,\lambda^*)$ to the problem \eqref{multivariateOptimization}. Namely, $\bfu^*$ satisfies the KKT conditions:
\begin{align}
\label{xsopt2}
A_i^\top\lambda^{\ast} &\in \partial f_i(\bfx_i^{\ast}),~\mbox{for}~i=1,\ldots,N,\\
\label{lsopt2}
A\bfx^*&=\sum_{i=1}^N A_i\bfx_i^{\ast} = c.
\end{align}
The optimality conditions \eqref{xsopt2} and \eqref{lsopt2} can be written in a more compact form by the following variational inequality \cite{he2013full}:
\beq \label{VI}
f(\bfx)-f(\bfx^*)+(\bfu-\bfu^*)^\top F(\bfu^*)\geq 0,~\forall u,
\eeq
where $f(\bfx):=\sum_i f_i(\bfx_i)$ and
\[
F(\bfu):=\begin{pmatrix}
-A_1^\top\lambda\\
\vdots\\
-A_N^\top\lambda\\
A\bfx-c
\end{pmatrix}.
\]
\end{assumption}

Let $\partial f_i(\bfx_i)$ denote the subdifferential of $f_i$ at $\bfx_i$:
\beq
\partial f_i(\bfx_i):=\left\{s_i\in\RR^{n_i}:~s_i^\top(\bfy_i-\bfx_i)\leq f_i(\bfy_i)-f_i(\bfx_i),~\forall
\bfy_i\in\dom f_i \right\}.
\eeq
We recall the following basic property for convex functions, which will be used several times in later sections.
\begin{lemma}[monotonicity of subdifferential]\label{lem:convex}
Under Assumption \ref{ass:convex}, for any $\bfx_i,\bfy_i\in \dom~f_i$, we have
\begin{equation} \label{convex}
(s_i-t_i)^\top(\bfx_i-\bfy_i)\geq 0,~\forall s_i\in\partial f_i(\bfx_i),~t_i\in\partial f_i(\bfy_i),
\end{equation}
for $i=1,2,\ldots,N$.
\end{lemma}
\begin{proof}
By definition, for any $s_i\in\partial f_i(\bfx_i)$ and $t_i\in\partial f_i(\bfy_i)$, we have
\begin{align*}
s_i^\top(\bfy_i-\bfx_i)&\leq f_i(\bfy_i)-f_i(\bfx_i),\\
~t_i^\top(\bfx_i-\bfy_i)&\leq f_i(\bfx_i)-f_i(\bfy_i).
\end{align*}
Adding these two inequalities together yields \eqref{convex}.
\end{proof}

In addition, we shall use an elementary lemma to improve the convergence rate from $O(1/k)$ to $o(1/k)$. Intuitively, the harmonic sequence $1/k$ is not summable, so a summable, nonnegative, monotonic sequence shall converge faster than $1/k$.
\begin{lemma} \label{lem:o}
If a sequence $\{a_k\}\subseteq\RR$ obeys: (1) $a_k\geq 0$; (2) $\sum_{k=1}^\infty a_k < +\infty$; (3) $a_k$ is monotonically non-increasing, then we have $a_k=o(1/k)$.
\end{lemma}
\begin{proof}
By the assumptions, we have
\[k\cdot a_{2k}\leq a_{k+1}+a_{k+2}+\cdots+a_{2k}\to 0\]
as $k\to +\infty$. Therefore, $a_k=o(1/k)$.
\end{proof}

\section{Convergence Analysis of the Proximal Jacobian ADMM}
In this section, we mainly study the convergence of  Proximal Jacobian ADMM (Algorithm \ref{Parallel_N}). We first show its convergence and then establish an $o(1/k)$ convergence rate in the same sense as in \cite{he2012non}. Furthermore, we discuss how to tune the parameter in order to make Proximal Jacobian ADMM more practical.

\subsection{Convergence}
To simplify the notation, we let
$$G_x:=\begin{pmatrix}
P_1+\rho A_1^\top A_1 &\quad &\quad\\
\quad &\ddots &\quad\\
\quad &\quad &P_N+\rho A_N^\top A_N
\end{pmatrix},
~G:=\begin{pmatrix}
G_x &\quad\\
\quad &\frac{1}{\gamma\rho}\vI
\end{pmatrix},
$$
where $\vI$ is the identity matrix of size $m \times m$.
In the rest of the section, we let $\{\bfu^k\}$ denote the sequence generated by  Proximal Jacobian ADMM from any initial point.
The analysis is based on bounding the error $\|\bfu^k-\bfu^*\|_G^2$ and estimating its decrease, motivated by the works \cite{he2012on,deng2012linear,he2013full}.
\begin{lemma} \label{lem:uopt}
For $k\ge 1$, we have
\beq \label{uopt}
\|\bfu^k-\bfu^*\|_G^2-\|\bfu^{k+1}-\bfu^*\|_G^2\geq h(\bfu^k,~\bfu^{k+1}),
\eeq
where
\beq
h(\bfu^k,~\bfu^{k+1}):=\|\bfx^k-\bfx^{k+1}\|_{G_x}^2+\frac{2-\gamma}{\rho\gamma^2}\|\lambda^k-\lambda^{k+1}
\|^2+\frac{2}{\gamma}(\lambda^k-\lambda^{k+1})^\top A(\bfx^k- \bfx^{k+1}).
\eeq
\end{lemma}
\begin{proof}
Recall that in Algorithm \ref{Parallel_N}, we solve the following $\bfx_i$-subproblem:
$$
\bfx_i^{k+1}=\argmin_{\bfx_i} f_i(\bfx_i)+\frac{\rho}{2}\left\|A_i\bfx_i+\sum_{j\neq i}A_j\bfx_j^k-c-\frac{
\lambda^k}{\rho}\right\|^2+\frac{1}{2}\|\bfx_i-\bfx_i^k\|_{P_i}^2.
$$
Its optimality condition is given by
\beq
\label{xiopt}
A_i^\top\left(\lambda^k-\rho(A_i\bfx^{k+1}_i+\sum_{j\neq i}A_j\bfx_j^k-c)\right)+P_i(\bfx_i^{k}-\bfx_i^{k+1})\in
 \partial f_i(\bfx_i^{k+1}).
\eeq
For convenience, we introduce $\hat{\lambda}:=\lambda^k-\rho(A\bfx^{k+1}-c)$. Then \eqref{xiopt} can be rewritten as
\beq
\label{xiopt2}
A_i^\top\left(\hat{\lambda}-\rho\sum_{j\neq i}A_j(\bfx_j^k-\bfx_j^{k+1})\right)+P_i(\bfx_i^{k}-\bfx_i^{k+1})\in
\partial f_i(\bfx_i^{k+1}).
\eeq
By Lemma~\ref{lem:convex}, it follows from \eqref{xsopt2} and \eqref{xiopt2} that
$$
\left\langle A_i(\bfx_i^{k+1}-\bfx_i^*),~\hat{\lambda}-\lambda^*-\rho\sum_{j\neq i}A_j(\bfx_j^k-\bfx_j^{k+1})\right\rangle +(\bfx_i^{k+1}-\bfx_i^*)^\top P_i(\bfx_i^{k}-\bfx_i^{k+1})\geq 0.
$$
Summing the above inequality over all $i$ and using the following equality for each $i$:
$$\sum_{j\neq i}A_j(\bfx_j^k-\bfx_j^{k+1})=A(\bfx^k-\bfx^{k+1})-A_i(\bfx_i^k-\bfx_i^{k+1}),$$
we obtain
\beq
\label{xiopt3}
\begin{split}
&\langle A(\bfx^{k+1}-\bfx^*),\hat{\lambda}-\lambda^*\rangle
+ \sum_{i=1}^{N}(\bfx_i^{k+1}-\bfx_i^*)^\top (P_i+\rho A_i^\top A_i)(\bfx_i^{k}-\bfx_i^{k+1})\\
&\geq \rho\langle A(\bfx^{k+1}-\bfx^*),A(\bfx^k-\bfx^{k+1})\rangle.
\end{split}
\eeq
Note that
$$A(\bfx^{k+1}-\bfx^*)=\frac{1}{\gamma\rho}(\lambda^k-\lambda^{k+1}),$$ and
$$\hat{\lambda}-\lambda^*=(\hat{\lambda}-\lambda^{k+1})+(\lambda^{k+1}-\lambda^*)=\frac{\gamma-1}{\gamma}
(\lambda^k-\lambda^{k+1})+(\lambda^{k+1}-\lambda^*).$$
With the above two equations,  the inequality \eqref{xiopt3} can be rewritten as
\beq
\begin{split}
&\langle \frac{1}{\gamma\rho}(\lambda^k-\lambda^{k+1}),\lambda^{k+1}-\lambda^*\rangle+\sum_{i=1}^{N}(\bfx_i^{k+1}-\bfx_i^*)^\top (P_i+\rho A_i^\top A_i)(\bfx_i^{k}-\bfx_i^{k+1})\\
&\geq \frac{1-\gamma}{\gamma^2\rho}\|\lambda^k-\lambda^{k+1}\|^2+\frac{1}{\gamma}(\lambda^k-\lambda^{k+1})^\top A(\bfx^k-\bfx^{k+1}),
\end{split}
\eeq
i.e.,
\beq
\label{xiopt4}
(\bfu^k-\bfu^{k+1})^\top G(\bfu^{k+1}-\bfu^*)
\geq \frac{1-\gamma}{\gamma^2\rho}\|\lambda^k-\lambda^{k+1}\|^2+\frac{1}{\gamma}(\lambda^k-\lambda^{k+1})^\top A(\bfx^k-\bfx^{k+1}).
\eeq
Since $\|\bfu^k-\bfu^*\|_G^2-\|\bfu^{k+1}-\bfu^*\|_G^2=2(\bfu^k-\bfu^{k+1})^\top G(\bfu^{k+1}-\bfu^*)+\|\bfu^k-\bfu^{k+1}\|_G^2$, using the above inequality \eqref{xiopt4} yields \eqref{uopt} immediately.
\end{proof}

\begin{lemma} \label{lem:contraction2}
Suppose the parameters $\rho$, $\gamma$ and $P_i~(i=1,2,\ldots,N)$ satisfy the following condition:
\begin{equation} \label{cond2}
\left\{
\begin{array}{l}
P_i\succ \rho(\frac{1}{\epsilon_i}-1)A_i^\top A_i,~i=1,2,\ldots,N\\
\sum_{i=1}^N \epsilon_i<2-\gamma,
\end{array}
\right.
\end{equation}
for some $\epsilon_i>0,~i=1,2,\ldots,N$. Then there exists some $\eta>0$ such that
\beq \label{h_bound}
h(\bfu^k,\bfu^{k+1})\geq~\eta\cdot\|\bfu^k-\bfu^{k+1}\|_G^2.
\eeq
Therefore,
\beq \label{contr2}
\|\bfu^k-\bfu^*\|_G^2-\|\bfu^{k+1}-\bfu^*\|_G^2\geq~\eta\cdot\|\bfu^k-\bfu^{k+1}\|_G^2.
\eeq
Condition \eqref{cond2} can be reduced to
\beq \label{cond2e}
P_i\succ \rho\left(\frac{N}{2-\gamma}-1\right)A_i^\top A_i,~i=1,2,\ldots,N,
\eeq
by letting each $\epsilon_i<\frac{2-\gamma}{N}$. In particular, for the following choices:\begin{itemize}
\item  $P_i=\tau_i\vI$ (standard proximal), condition \eqref{cond2e} becomes
$\tau_i>\rho\left(\frac{N}{2-\gamma}-1\right)\|A_i\|^2$; 
\item $P_i=\tau_i\vI-\rho A_i^\top A_i$ (prox-linear),
condition \eqref{cond2e} becomes $\tau_i>\frac{\rho
N}{2-\gamma}\|A_i\|^2$.
\end{itemize}
\end{lemma}
\begin{proof}
By the Cauchy-Schwarz inequality,
\begin{align}
\nonumber
\frac{2}{\gamma}(\lambda^k-\lambda^{k+1})^\top A(\bfx^k-\bfx^{k+1})
=&\sum_{i=1}^{N}\frac{2}{\gamma}(\lambda^k-\lambda^{k+1})^\top A_i(\bfx_i^k-\bfx_i^{k+1})\\
\label{cauchy}
\geq&-\sum_{i=1}^{N}\left(\frac{\epsilon_i}{\rho\gamma^2}\|\lambda^k-\lambda^{k+1}
\|^2+\frac{\rho}{\epsilon_i}
\|A_i(\bfx_i^k-\bfx_i^{k+1})\|^2\right)
\end{align}
Then we have
\beq \label{h_bound2}
h(\bfu^k,\bfu^{k+1})
\geq \sum_{i=1}^{N}\|\bfx_i^k-\bfx_i^{k+1}\|_{P_i+\rho A_i^\top A_i-\frac{\rho}{\epsilon_i}A_i^\top A_i}^2+\frac
{2-\gamma-\sum_{i=1}^N \epsilon_i}{\rho\gamma^2}\|\lambda^k-\lambda^{k+1}\|^2.
\eeq
The condition \eqref{cond2} guarantees that $P_i+\rho A_i^\top A_i-\frac{\rho}{\epsilon_i}A_i^\top A_i\succ 0$
and $2-\gamma-\sum_{i=1}^N \epsilon_i>0$. Therefore, we must have \eqref{h_bound} for some $\eta>0$.
By Lemma \ref{lem:uopt}, \eqref{contr2} follows immediately.
\end{proof}

Lemma \ref{lem:contraction2} shows that the iterative sequence $\{\bfu^k\}$ is \textit{strictly contractive}. It follows that the error $\|\bfu^k-\bfu^*\|_G^2$ is monotonically non-increasing and thus converging, as well as $\|\bfu^k-\bfu^{k+1}\|_G^2\to 0$.
The convergence of the algorithm follows immediately from the standard analysis for contraction methods (see, e.g., \cite{he1997class}). We omit the details of the proof for the sake of brevity.
\begin{theorem}
Suppose the parameters in Algorithm \ref{Parallel_N} satisfy the condition \eqref{cond2}. Then the sequence $\{\bfu^k\}$ generated by Algorithm \ref{Parallel_N} converges to a solution $\bfu^*$ to the problem \eqref{multivariateOptimization}.
\end{theorem}

\subsection{Rate of Convergence}
Next, we shall establish the $o(1/k)$ convergence rate of Proximal Jacobian ADMM.
We use the quantity $\|\bfu^k-\bfu^{k+1}\|_{G'}^2$ as a measure of the convergence
rate motivated by \cite{he2012non,he2013full}. Here, we define the matrix $G'$ by
$$
G':=\begin{pmatrix}
G'_x &\quad\\
\quad &\frac{1}{\gamma\rho}\vI
\end{pmatrix}\quad\mbox{and}\quad
G'_x:=G_x-\rho A^\top A.$$

\begin{theorem}\label{thm:paraADM_rate}
If $G'_x\succeq0$ and \eqref{cond2} holds, then
$$\|\bfu^{k} - \bfu^{k+1}\|_{G'}^2=o(1/k),$$
and, thus,
$$\|\bfx^k-\bfx^{k+1}\|_{G'_x}^2=o(1/k)\quad\mbox{and}\quad\|\lambda^k-\lambda^{k+1}\|^2=o(1/k).$$
\end{theorem}

We need the following monotonic property of the iterations:
\begin{lemma} \label{lem:mono2}
If $G'_x\succeq0$ and $0<\gamma<2$,
then
\beq \label{monotone_N}
\|\bfu^k-\bfu^{k+1}\|_{G'}^2\leq \|\bfu^{k-1}-\bfu^{k}\|_{G'}^2.
\eeq
\end{lemma}
\begin{proof}
Let $\Dx^{k+1}_i=\bfx_i^k-\bfx_i^{k+1}, i=1, \ldots, N$, $\Dx^{k+1}=\bfx^k -\bfx^{k+1}$,
and $\Dlambda^{k+1}=\lambda^k-\lambda^{k+1}$.
By Lemma~\ref{lem:convex}, the optimality conditions \eqref{xiopt2} at $k$-th and $(k+1)$-th iterations yield
\beq
\langle A_i\Dx_i^{k+1},\Dlambda^k-\rho A\Dx^{k+1}-\rho
\sum_{j\neq i}A_j(\Dx_j^k-\Dx_j^{k+1})\rangle+
(\Dx_i^{k+1})^\top P_i(\Dx_i^{k}-\Dx_i^{k+1})\geq 0.
\eeq
Summing up over all $i$ and rearranging the terms, we have
\beq
\langle A\Dx^{k+1},\Dlambda^k\rangle\geq \|\Dx^{k+1}\|_{G_x}^2-(\Dx^k)^\top(G_x-
\rho A^\top A)\Dx^{k+1}.
\eeq
Since $G'_x:=G_x-\rho A^\top A\succeq 0$, we have
\beq
2(\Dx^k)^\top(G_x-\rho A^\top A)\Dx^{k+1}\leq \|\Dx^k\|_{G'_x}^2+\|\Dx^k\|_{G'_x}^2,
\eeq
and thus
\begin{align}
\nonumber
2\langle A\Dx^{k+1},\Dlambda^k\rangle &\geq \|\Dx^{k+1}\|_{2G_x-G'_x}^2-\|\Dx^k\|_{G'_x}^2\\
&=\|\Dx^{k+1}\|_{G_x+\rho A^\top A}^2-\|\Dx^k\|_{G'_x}^2.
\end{align}
Note that $\Dlambda^{k+1}=\Dlambda^k-\gamma\rho A\Dx^{k+1}$. It follows that
\begin{align}
\nonumber
\frac{1}{\gamma\rho}\|\Dlambda^k\|^2-\frac{1}{\gamma\rho}\|\Dlambda^{k+1}\|^2
=&~2\langle A\Dx^{k+1},\Dlambda^k\rangle-\gamma\rho\|A\Dx^{k+1}\|^2\\
\geq&~ \|\Dx^{k+1}\|_{G_x+(1-\gamma)\rho A^\top A}^2-\|\Dx^k\|^2_{G'_x},
\end{align}
i.e.,
\beq
\begin{split}
&(\|\Dx^k\|^2_{G'_x}+\frac{1}{\gamma\rho}\|\Dlambda^k\|^2)-(\|\Dx^{k+1}\|_{G'_x}+
\frac{1}{\gamma\rho}\|\Dlambda^{k+1}\|^2)\\
&\geq \|\Dx^{k+1}\|_{(2-\gamma)\rho A^\top A}^2\geq0,
\end{split}
\eeq
which completes the proof.
\end{proof}

\begin{proof}[Proof of Theorem \ref{thm:paraADM_rate}]
By Lemma \ref{lem:contraction2}, we have
\beq \label{contr_G'}
\|\bfu^{k}-\bfu^*\|^2_G - \|\bfu^{k+1}-\bfu^*\|^2_G
\geq \eta\|\bfu^k-\bfu^{k+1}\|^2_G
\geq \eta\|\bfu^k-\bfu^{k+1}\|^2_{G'}.
\eeq
Summing \eqref{contr_G'} over $k$ gives
\beq
\sum_{k=1}^{\infty} \|\bfu^k-\bfu^{k+1}\|^2_{G'} < \infty.
\eeq
On the other hand, Lemma \ref{lem:mono2} implies the monotone non-increasing of $\|\bfu^k-\bfu^{k+1}\|^2_{G'}$. By Lemma \ref{lem:o}, we have $\|\bfu^{k} - \bfu^{k+1}\|_{G'}^2=o(1/k)$, which completes the proof.
\end{proof}

\subsection{Adaptive Parameter Tuning} \label{sec:adaptive}
The parameters satisfying the condition \eqref{cond2} may be rather conservative because the Cauchy-Schwarz inequality \eqref{cauchy} for bounding $h(\bfu^k,\bfu^{k+1})$ is usually very loose.
In practice, we can exactly compute $h(\bfu^k,\bfu^{k+1})$ at very little extra cost, instead of using the bound in \eqref{h_bound2}.
Based on the value of $h(\bfu^k,\bfu^{k+1})$, we thereby propose a practical strategy for adaptively adjusting the matrices $\{P_i\}$:

\begin{algorithm}[H]
Initialize with small $P^0_i\succeq 0~(i=1,2,\ldots, N)$ and a small $\eta>0$\;
\For{$k=1,2,\ldots$}{
\eIf{$h(\bfu^{k-1},\bfu^{k}) > \eta\cdot \|\bfu^{k-1}-\bfu^{k}\|_G^2$}
{$P_i^{k+1}\leftarrow P_i^k,~\forall i$\;}
{Increase $P_i$: $P_i^{k+1}\leftarrow \alpha_i P_i^k+\beta_i Q_i ~(\alpha_i>1, ~\beta_i\ge 0, ~Q_i\succ 0),\forall i$\;
 Restart: $\bfu^{k}\leftarrow \bfu^{k-1}$\;}
}
\end{algorithm}

The above strategy starts with relatively small proximal terms and gradually increase them.
By Lemma \ref{lem:contraction2}, we know that when the parameters $\{P_i\}$ are large enough for \eqref{cond2} to hold, the condition \eqref{h_bound} will be satisfied (for sufficiently small $\eta$). Therefore, the adjustment of $\{P_i\}$ cannot occur infinite times.
After a finite number of iterations, $\{P_i\}$ will remain constant and the contraction property \eqref{contr2} of the iterations will hold.
Therefore, the convergence of such an adaptive parameter tuning scheme follows immediately from our previous analysis.

\begin{theorem}
Suppose the matrices $P_i~(i=1,2,\ldots,N)$ in Algorithm \ref{Parallel_N} are adaptively adjusted using the above scheme. Then the algorithm converges to a solution to the problem \eqref{multivariateOptimization}.
\end{theorem}

Empirical evidence shows that the paramters $\{P_i\}$ typically adjust themselves only during the first few iterations and then remain constant afterwards. Alternatively, one may also decrease the parameters after every few iterations or after they have not been updated for a certain number of iterations. But the total times of decrease should be bounded to guarantee convergence. By using this adaptive strategy, the resulting paramters $\{P_i\}$ are usually much smaller than those required by the condition \eqref{cond2}, thereby leading to substantially faster convergence in practice.

\section{Numerical Experiments}
In this section, we present numerical results to compare the performance of the following parallel splitting algorithms:
\begin{itemize}
\item \textbf{Prox-JADMM}: proposed Proximal Jacobian ADMM (Algorithm \ref{Parallel_N});
\item \textbf{VSADMM}: Variable Splitting ADMM (Algorithm \ref{alg:vsadmm});
\item \textbf{Corr-JADMM}: Jacobian ADMM with correction steps \cite{he2013full}. At every iteration, it first generates a ``predictor" $\tilde{\bfu}^{k+1}$ by an iteration of Jacobian ADMM (Algorithm \ref{alg:Jacobian-admm}) and then corrects $\tilde{\bfu}^{k+1}$ to generate the new iterate by:
\begin{equation}
\bfu^{k+1}=\bfu^k-\alpha_k(\bfu^k-\tilde{\bfu}^{k+1}),
\end{equation}
where $\alpha_k>0$ is a step size. In our experiments, we adopt the dynamically updated step size $\alpha_k$ according to \cite{he2013full}, which is shown to converge significantly faster than using a constant step size, though updating the step size requires extra computation.
\item \textbf{YALL1}:  one of the state-of-the-art solvers for the $\ell_1$-minimization problem.
\end{itemize}

In Section \ref{sec:Exchange_Problem} and \ref{sec:l1min}, all of the numerical experiments are run in MATLAB (R2011b) on a workstation with an Intel Core i5-3570 CPUs (3.40GHz) and 32 GB of RAM. Section \ref{sec:large_test} gives two very large instances that are solved by a C/MPI implementation on  Amazon Elastic Compute Cloud (EC2).

\subsection{Exchange Problem}
\label{sec:Exchange_Problem}
Consider a network of $N$ agents that exchange $n$ commodities.
Let $\bfx_i\in\RR^n~(i=1,2,\ldots,N)$ denote the amount of commodities that are exchanged among the $N$ agents.
Each agent $i$ has a certain cost function $f_i:\RR^n\rightarrow \RR$.
The exchange problem (see, e.g., \cite{boyd2010distributed} for a review) is given by
\beq \label{exchange}
\min_{\{\bfx_i\}} ~\sum_{i=1}^N~f_i(\bfx_i)\quad\st ~\sum_{i=1}^N \bfx_i=0,
\eeq
which minimizes the total cost among $N$ agents subject to an equilibrium constraint on the commodities.
This is a special case of \eqref{multivariateOptimization} where $A_i=\vI$ and $c=0$.

We consider quadratic cost functions $f_i(\bfx_i):=\frac{1}{2}\|C_i\bfx_i-d_i\|^2$, where $C_i\in\RR^{p\times n}$ and $d_i\in\RR^p$.
Then all the compared algorithms solve the following type of subproblems at every iteration:
\beq
\bfx_i^{k+1}=\argmin_{\bfx_i}~\frac{1}{2}\|C_i\bfx_i-d_i\|^2+\frac{\rho}{2}\|\bfx_i-b_i^k\|^2,~\forall i=1,2,\ldots,N,
\eeq
except that Prox-JADMM also adds a proximal term $\frac{1}{2}\|\bfx_i-\bfx_i^k\|_{P_i}^2$.
Here $b_i^k\in\RR^{m}$ is a vector independent of $\bfx_i$ and takes different forms in different algorithms.
For Prox-JADMM, we simply set $P_i=\tau_i\vI~(\tau_i>0)$. Clearly, each $\bfx_i$-subproblem is a quadratic program that can be computed efficiently using various methods.

In our experiment, we randomly generate $\bfx^*_i,~i=1,2,\ldots,N-1$, following the standard Gaussian distribution, and let $\bfx^*_N=-\sum_{i=1}^{N-1}\bfx^*_i$.
Matrices $C_i$ are random Gaussian matrices, and  vectors $d_i$ are computed by $d_i=C_i\bfx_i^*$.
Apparently, $\bfx^*$ is  a solution (not necessarily unique) to \eqref{exchange}, and the optimal objective value is 0.

The penalty parameter $\rho$ is set to be 0.01, 1 and 0.01 for Prox-JADMM, VSADMM and Corr-JADMM, respectively.
They are nearly optimal for each algorithm, picked out of a number of different values.
Note that the parameter for VSADMM is quite different from the other two algorithms because it has different constraints due to the
variable splitting. For Prox-JADMM, the proximal parameters are initialized by $\tau_i=0.1(N-1)\rho$ and adaptively updated by the strategy in
Subsection \ref{sec:adaptive}; the parameter $\gamma$ is set to be 1.

The size of the test problem is set to be $n=100,~N=100,~p=80$. Letting all the algorithms run 200 iterations, we plot their objective value
$\sum_{i=1}^N f_i(\bfx_i)$ and residual $\|\sum_{i=1}^N \bfx_i\|_2$. Note that the per-iteration cost (in terms of both computation and communication) is roughly the same for all the compared algorithms.
Figure \ref{fig:exchange} shows the comparison result, which is averaged over 100 random trials. We can see that Prox-JADMM is clearly the fastest one among the compared algorithm.

\begin{figure}[htbp]
\centering
\includegraphics[width=0.95\textwidth]{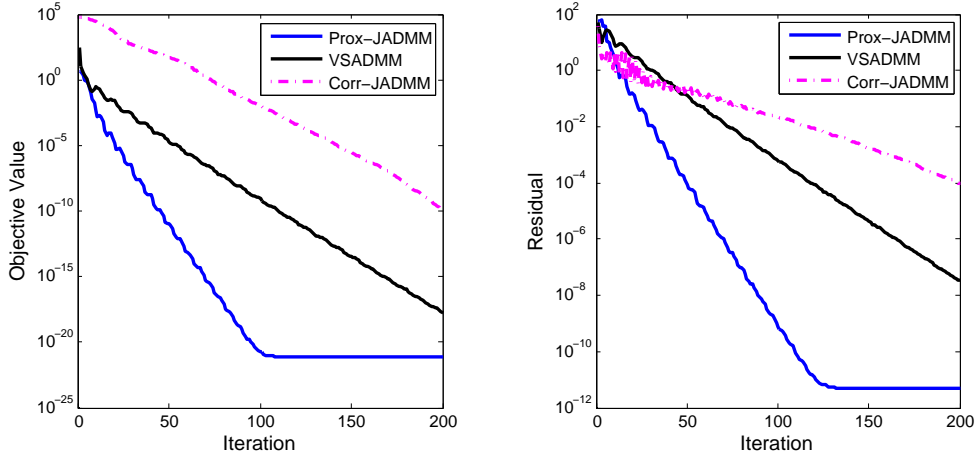}
\caption{Exchange problem $(n=100,~N=100,~p=80)$.}
\label{fig:exchange}
\end{figure}

\subsection{$\ell_1$-minimization}\label{sec:l1min}
We consider the $\ell_1$-minimization problem for finding sparse solutions of an underdetermined linear system:
\beq \label{l1min}
\min_{\bfx}~\|\bfx\|_1\quad\st ~A\bfx = c,
\eeq
where $\bfx\in\RR^n$, $A\in \RR^{m\times n}$ and $c\in\RR^m$ ($m<n$).
It is also known as the \textit{basis pursuit} problem, which has been widely used in compressive sensing, signal and image processing, statistics, and machine learning.
Suppose that the data is partitioned into $N$ blocks: $\bfx=[\bfx_1,\bfx_2,\ldots,\bfx_N]$ and $A=[A_1,A_2,\ldots,A_N]$. Then the problem \eqref{l1min} can be written in the form of (\ref{multivariateOptimization}) with $f_i(\bfx_i)= \|\bfx_i\|_1$.

In our experiment, a sparse solution $\bfx^*$ is randomly generated with $k~(k\ll n)$ nonzeros drawn from the standard Gaussian distribution. Matrix $A$ is also randomly generated from the standard Gaussian distribution, and it is partitioned evenly into $N$ blocks. The vector $c$ is then computed by $c=A\bfx^*+\eta$, where $\eta\sim \cN(0,~\sigma^2\vI)$ is Gaussian noise with standard deviation $\sigma$.

Prox-JADMM solves the $\bfx_i$-subproblems with $P_i=\tau_i \vI-\rho A_i^\top A_i~(i=1,2,\ldots,N)$ as follows:
\begin{align} \label{l1-proxJADMM}
\nonumber
\bfx_i^{k+1} &=\argmin_{\bfx_i}~\|\bfx_i\|_1 +\frac{\rho}{2}\left\|A_i\bfx_i+\sum_{j\ne i}A_j\bfx_j^k-c-\frac{\lambda^k}{\rho}\right\|^2 +\frac{1}{2}\left\|\bfx_i-\bfx_i^k\right\|_{P_i}^2\\
&=\argmin_{\bfx_i}~\|\bfx_i\|_1+\left\langle \rho A_i^\top\left(A\bfx^k-c-
\frac{\lambda^k}{\rho}\right),\bfx_i\right\rangle+\frac{\tau_i}{2}
\left\|\bfx_i-\bfx_i^k\right\|^2.
\end{align}
Here, we choose the prox-linear $P_i$'s  to linearize the original subproblems, and thus \eqref{l1-proxJADMM} admits a simple closed-form solution by the \textit{shrinkage} (or \textit{soft-thresholding}) formula.
The proximal parameters are initialized as $\tau_i=0.1N\rho$ and are adaptively updated by the strategy discussed in Section \ref{sec:adaptive}.

Recall that VSADMM needs to solve the following $\bfx_i$-subproblems:
\begin{align}
\bfx_i^{k+1}=\argmin_{\bfx_i} ~\|\bfx_i\|_1+\frac{\rho}{2}\left\|A_i\bfx_i-\bfz_i^{k+1}-\frac{c}{N}-\frac{\lambda^k_i}{\rho}\right\|^2.
\end{align}
Such subproblems are not easily computable, unless $\bfx_i$ is a scalar (i.e., $n_i=1$) or $A_i^\top A_i$ is a diagonal matrix.
Instead, we solve the subproblems approximately using the {prox-linear} approach:
\begin{align} \label{l1_prox}
\bfx_i^{k+1}=\argmin_{\bfx_i}~\|\bfx_i\|_1+\left\langle \rho A_i^\top\left(A_i\bfx^k_i-\bfz_i^{k+1}-\frac{c}{N}-
\frac{\lambda^k_i}{\rho}\right), \bfx_i-\bfx_i^k\right\rangle+\frac{\tau_i}{2}\left\|\bfx_i-\bfx_i^k\right\|^2,
\end{align}
which can be easily computed by the shrinkage operator. We set $\tau_i=1.01\rho\|A_i\|^2$ in order to guarantee the convergence, as suggested in \cite{wang2013solving}.

Corr-JADMM solves the following $\bfx_i$-subproblems in the ``prediction" step:
\begin{align} \label{l1-CorrJADMM}
\tilde{\bfx}_i^{k+1}=\argmin_{\bfx_i}~\|\bfx_i\|_1+\frac{\rho}{2}\left\|A_i\bfx_i+\sum_{j\ne i}A_j
\bfx_j^k-c-\frac{\lambda^k}{\rho}\right\|^2.
\end{align}
Because the correction step in \cite{he2013full} is based on exact minimization of the subproblems, we do not apply the prox-linear approach to solve the subproblems approximately. Instead, we always partition $\bfx$ into scalar components (i.e., $N=n$) so that the subproblems \eqref{l1-CorrJADMM} can still be computed exactly. The same penalty parameter $\rho=10/\|c\|_1$ is used for the three algorithms. It is nearly optimal for each algorithm, selected out of a number of different values.

In addition, we also include the YALL1 package \cite{yang2011alternating} in the experiment, which is one of the state-of-the-art solvers for
$\ell_1$ minimization.
Though YALL1 is not implemented in parallel, the major computation of its iteration is matrix-vector multiplication by $A$ and $A^\top$, which can be easily parallelized (see \cite{peng2013parallel}).
Since all the compared algorithms have roughly the same amount of per-iteration cost (in terms of both computation and communication), we simply let all the algorithms run for a fixed number of iterations and plot their relative error $\frac{\|\bfx^k-\bfx^*\|_2}{\|\bfx^*\|_2}$.

\begin{figure}[htbp]
\centering
    \subfigure[Noise-free $(\sigma = 0)$]{
        \includegraphics[width=0.47\textwidth]{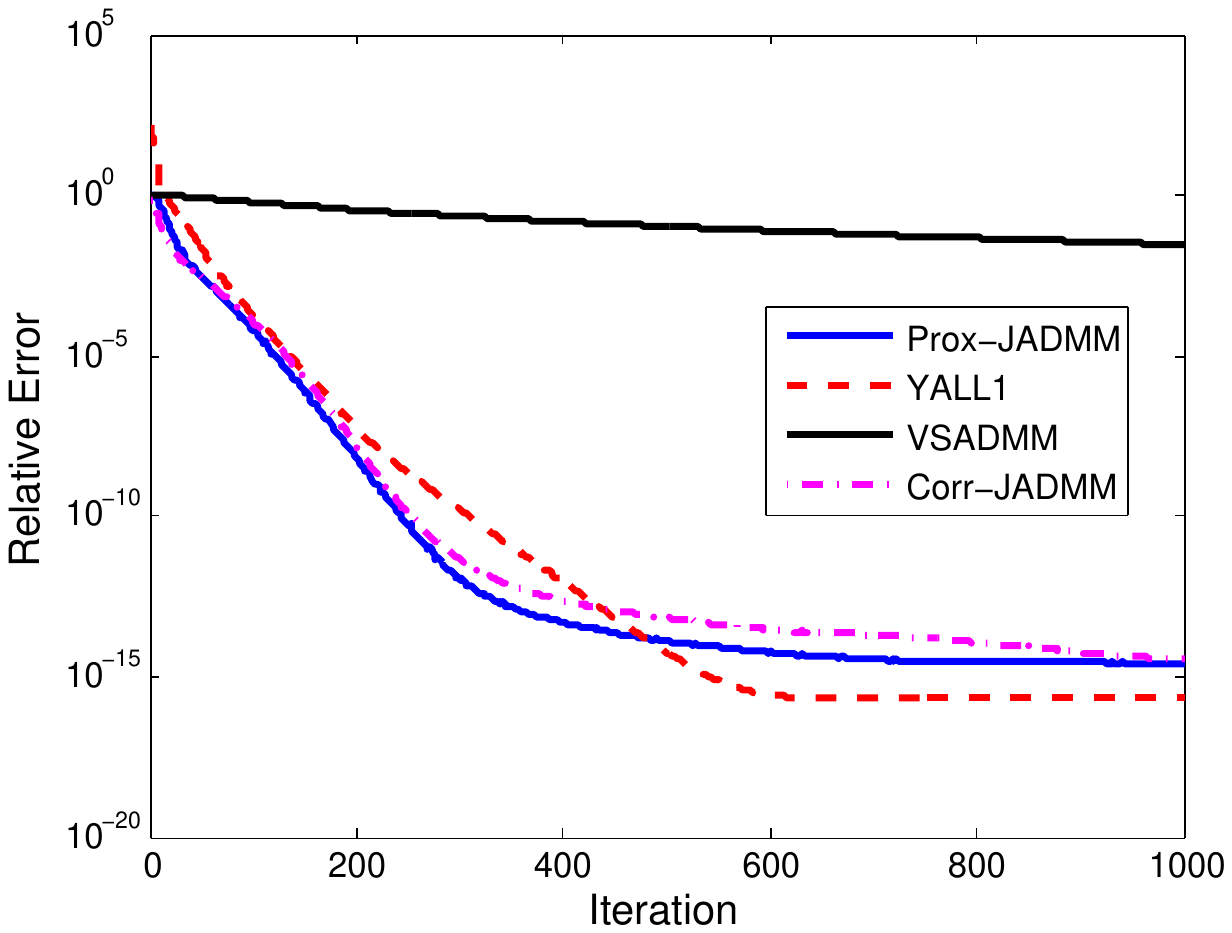}
        \label{subfig:L1a}
    }
    \subfigure[Noise added $(\sigma= 10^{-3})$]{
        \includegraphics[width=0.47\textwidth]{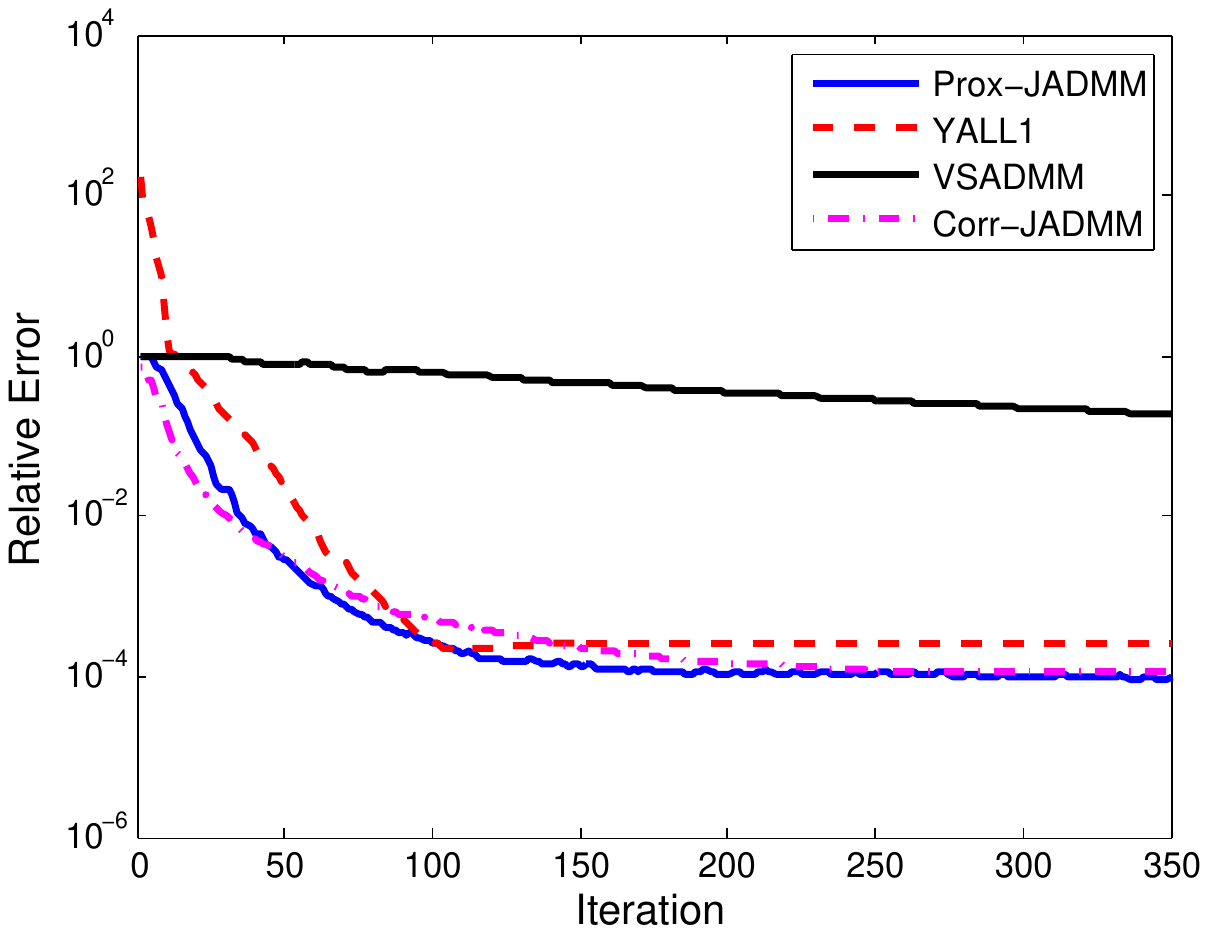}
        \label{subfig:L1b}
    }
     \caption{$\ell_1$-problem $(n=1000,~m=300,~k=60)$.}
     \label{fig:L1}
\end{figure}

Figure \ref{fig:L1} shows the comparison result where $n=1000,~m=300,~k=60$ and the standard deviation of noise $\sigma$ is set to be 0 and $10^{-3}$, respectively.
For Prox-JADMM and VSADMM, we set $N=100$; for Corr-JADMM, we set $N=1000$. The results are average of 100 random trials.
We can see that Prox-JADMM and Corr-JADMM achieve very close performance and are the fastest ones among the compared algorithms. YALL1 also shows competitive performance. However, VSADMM is far slower than the others, probably due to inexact minimization of the subproblems and the conservative proximal parameters.

\subsection{Distributed Large-Scale $\ell_1$-Minimization}
\label{sec:large_test}
In previous subsections, we described the numerical simulation of a distributed implementation of Proximal Jacobian ADMM that was  carried out  in Matlab. We now turn to realistic distributed examples and solve two very large instances of the $\ell_1$-minimization problem \eqref{l1min} using a C code with MPI for inter-process communication and the GNU Scientific Library (GSL) for BLAS operations. The experiments are carried out  on Amazon's Elastic Compute Cloud (EC2).

We generate two test instances as shown in Table \ref{tab:largeData}. Specifically, a sparse solution $\bfx^*$ is randomly generated with $k$ nonzeros drawn from the standard Gaussian distribution. Matrix $A$ is also randomly generated from the standard Gaussian distribution with $m$ rows and $n$ columns, and it is partitioned evenly into $N=80$ blocks. Vector $c$ is then computed by $c=A\bfx^*$. Note that $A$ is dense and has double precision. For Test 1 it requires over $150$ GB of RAM and has 20 billion nonzero entries, and for Test 2 it requires over $337$GB of RAM. Those two tests are far too large to process on a single PC or workstation. We want to point out that we cannot find a dataset of similar or larger size in the public domain. We are willing to test our a larger problem per reader's request.
\begin{table}[!htbp]
\centering
 \caption{\label{tab:largeData}Two large datasets}
 \begin{tabular}{lcccc}
  \toprule
  & $m$ & $n$ &$k$ & RAM\\
  \midrule
dataset 1 & $1.0 \times 10^5$ & $2.0 \times 10^5$& $2.0 \times 10^3$ & $150$GB \\
dataset 2 & $1.5 \times 10^5$ & $3.0 \times 10^5$ & $3.0 \times 10^3$ & $337$GB \\
  \bottomrule
 \end{tabular}
\end{table}

We solve the problem using a cluster of 10 machines, where each machine is a ``memory-optimized instance" with $68$ GB RAM and 1 eight-core Intel Xeon E5-2665 CPU. Those instances run Ubuntu 12.04 and are connected with 10 Gigabit ethernet network. Since each  has 8 cores, we run the code with 80 processes so that each process runs on its own core. Such a  setup is charged for under $\$17$ per hour.

We solve the large-scale $\ell_1$ minimization problems with a C implementation that matches the Matlab implementation in the previous section. The implementation consists of a single file of C code of about 300 lines, which is available for download on our personal website.

\begin{table}[!htbp]
\centering
 \caption{\label{tab:largeDataTime}Time results for large scale $\ell_1$ minimization examples}
 \begin{tabular}{lrrrrrr}
  \toprule
  &\multicolumn{3}{ c }{$150$GB Test } &  \multicolumn{3}{ c }{$337$GB Test }\\
  & \cline{1-6}
  & Itr & Time(s) & Cost(\$) & Itr & Time(s) & Cost(\$)\\
  \midrule
Data generation & --& 44.4  &0.21 & -- &99.5 & 0.5\\
CPU per iteration & --&  1.32 & --   &--& 2.85 & --   \\
Comm. per iteration &--& 0.07  & --  &--& 0.15  & --\\
Reach $10^{-1}$ &23  &30.4  &0.14 &27 &  79.08 & 0.37 \\
Reach $10^{-2}$ &30  &39.4 &0.18  &39 & 113.68& 0.53  \\
Reach $10^{-3}$ &86  &112.7&0.53  &84 & 244.49& 1.15  \\
Reach $10^{-4}$ &234 &307.9&1.45  &89 & 259.24& 1.22  \\
\bottomrule
 \end{tabular}
\end{table}

The breakdown of the wall-clock time is summarized in Table \ref{tab:largeDataTime}. We can observe that Jacobian ADMM is very efficient in obtaining a relative low accuracy, which is usually sufficient for large-scale problems. We want to point out that the basic BLAS operations in our implantation can be further improved by using other libraries such as  hardware-optimized BLAS libraries produced by ATLAS, Armadillo, etc. Those libraries might lead to several times of speedup\footnote{http://nghiaho.com/?p=1726}. We use GSL due to its ease of use, so the code can be easily adapted for solving similar problems.

\section{A Sufficient Condition for Convergence of Jacobian ADMM}\label{sec:suf}
In this section, we provide a sufficient condition to guarantee the
convergence of  Jacobian ADMM (Algorithm \ref{alg:Jacobian-admm}), which does not use either proximal terms or correction steps.
The condition only depends on the coefficient matrices $A_i$, without imposing further assumptions on the objective functions $f_i$ or  the penalty parameter $\rho$.
For the Gauss-Seidel ADMM (Algorithm \ref{alg:admm}), a sufficient condition for convergence
is provided in \cite{chen2013direct} for the special case $N=3$,
assuming two of the three coefficient matrices are
orthogonal. Our condition does not require exact orthogonality.
Instead, we mainly assume that the matrices
$A_i,~i=1,2,\ldots, N$ are mutually ``near-orthogonal" and have full column-rank.

\begin{theorem}
\label{thm:Jacobian_conv}
Suppose that there exists $\delta>0$ such that
\beq \label{JADMM_cond}
\|A_i^\top A_j\|\le \delta,~\forall~i\ne j,~\mbox{and }
\lambda_{min}(A_i^\top A_i)> 3(N-1)\delta,~\forall~i,
\eeq
where $\lambda_{min}(A_i^\top A_i)$ denotes the smallest eigenvalue of $A_i^\top A_i$.
Then the sequence $\{\bfu^k\}$ generated by Algorithm \ref{alg:Jacobian-admm} converges to a solution $\bfu^*$ to the problem \eqref{multivariateOptimization}.
\end{theorem}

The proof technique is motivated by the contraction analysis of the sequence $\{\bfu^k\}$ under some $G$-norm (e.g., see \cite{he2012on,deng2012linear,he2013full}).
To prove the theorem, we first need the following lemma:

\begin{lemma} \label{lem:uopt0}
Let
$$G_0:=\begin{pmatrix}
\rho A_1^\top A_1 &\quad &\quad &\quad\\
\quad &\ddots &\quad &\quad\\
\quad &\quad &\rho A_N^\top A_N &\quad\\
\quad &\quad &\quad &\frac{1}{\rho}\vI
\end{pmatrix},$$
where $\vI$ is the identity matrix of size $m \times m$.
For $k\ge 1$, we have
\beq \label{uopt0}
\|\bfu^k-\bfu^*\|_{G_0}^2-\|\bfu^{k+1}-\bfu^*\|_{G_0}^2\geq h_0(\bfu^k,~\bfu^{k+1}),
\eeq
where
\beq \label{h0}
h_0(\bfu^k,~\bfu^{k+1}):=\|\bfu^k-\bfu^{k+1}\|_{G_0}^2+2(\lambda^k-\lambda^{k+1})^\top A(\bfx^k- \bfx^{k+1}).
\eeq
\end{lemma}
This lemma follows directly from Lemma \ref{lem:uopt} since it is a special case with $\gamma=1$ and $P_i=0,~\forall i$. Now we are ready to prove the theorem.

\begin{proof}[Proof of Theorem \ref{thm:Jacobian_conv}]
By the assumption $\|A_i^\top A_j\|\le \delta,~i\ne j$, we have
\beq \label{cauchy2}
\left|\sum_{i\ne j}\langle A_i\bfa_i,A_j\bfb_j\rangle\right|\le \sum_{i\ne j} \delta\|\bfa_i\|\|\bfb_j\|\le
\frac{\delta}{2}(N-1)(\|\bfa\|^2+\|\bfb\|^2),~\forall~\bfa,\bfb
\eeq
To simplify the notation, we let
\beq
\bfa_i^k:=\bfx_i^k-\bfx_i^*,~i=1,2,\ldots,N.
\eeq
Note that
\[\lambda^k-\lambda^{k+1} = \rho A\bfa^{k+1},~\bfx^k-\bfx^{k+1}=\bfa^k-\bfa^{k+1}.\]
Then, we can rewrite \eqref{h0} as
\begin{align}
\frac{1}{\rho}h_0(\bfu^k-\bfu^{k+1}) &= \sum_i \|A_i(\bfa_i^k-\bfa_i^{k+1})\|^2 + \|A\bfa^{k+1}\|^2 + 2\langle A
\bfa^{k+1},~A(\bfa^k-\bfa^{k+1}) \rangle\\
&= \sum_i\|A_i\bfa_i^k\|^2 + 2\sum_{i\ne j} \langle A_i\bfa_i^{k+1},A_j\bfa_j^k\rangle-\sum_{i\ne j}\langle A_i
\bfa_i^{k+1},A_j\bfa_j^{k+1}\rangle\\
\label{h0_ineq}
&\geq \sum_i\|A_i\bfa_i^k\|^2 - (N-1)\delta (\|\bfa^{k+1}\|^2+\|\bfa^k\|^2) -(N-1)\delta\|\bfa^{k+1}\|^2\\
&= \sum_i\|A_i\bfa_i^k\|^2 - (N-1)\delta\|\bfa^k\|^2 - 2(N-1)\delta\|\bfa^{k+1}\|^2,
\end{align}
where the inequality \eqref{h0_ineq} comes from \eqref{cauchy2}. By Lemma \ref{lem:uopt0}, we have
\begin{align} \label{uopt_e}
\nonumber
&\|\bfu^k-\bfu^*\|_{G_0}^2- 2(N-1)\delta \rho\|\bfa^{k}\|^2\\ \geq&\|\bfu^{k+1}-\bfu^*\|_{G_0}^2- 2(N-1)
\delta \rho\|\bfa^{k+1}\|^2+\rho\sum_i\|A_i\bfa_i^k\|^2 - 3(N-1)\delta \rho\|\bfa^{k}\|^2.
\end{align}
We further simplify \eqref{uopt_e} as
\beq \label{uopt_bd}
\vb^k-\vb^{k+1}\geq \vd^k,
\eeq
where the sequences $\{\vb^k\}$ and $\{\vd^k\}$ are defined by
\begin{align}
\label{bk}
\bfb^k&:= \|\bfu^k-\bfu^*\|_{G_0}^2- 2(N-1)\delta\rho\|\bfa^{k}\|^2,\\
\vd^k&:=\rho\sum_i\|A_i\bfa_i^k\|^2 - 3(N-1)\delta \rho\|\bfa^{k}\|^2.
\end{align}
By the definition of $G_0$, we have
\beq \label{bk2}
\vb^k = \rho \sum_i \|A_i\bfa_i^k\|^2-2(N-1)\delta \rho \|\bfa_i^k\|^2+\frac{1}{\rho}\|\lambda^k-\lambda^*\|^2.
\eeq
Since we assume $\lambda_{\min}(A_i^\top A_i)>3(N-1)\delta$, it follows that
\beq
\|A_i\bfa_i^k\|^2 \geq 3(N-1)\delta\|\bfa_i^{k}\|^2,~\forall i.
\eeq
Then it is easy to see that $\vb^k\geq 0$ and $\vd^k \geq 0$.
By \eqref{uopt_bd}, the nonnegative sequence $\{\vb^k\}$ is monotonically non-increasing. Hence, $\{\vb^k\}$
converges to some $\vb^*\geq 0$. By \eqref{uopt_bd}, it also follows that $\vd^k \to 0$. Therefore, $\va^k \to
0$, i.e., $\bfx^k \to \bfx^*$.

Next we show $\lambda^k\to\lambda^*$. By taking limit of \eqref{bk2} and using $\va^k \to 0$, we have
\beq \label{b*}
\vb^*=\lim_{k\to\infty}\vb^k=\lim_{k\to\infty}\frac{1}{\rho}\|\lambda^k-\lambda^*\|^2.
\eeq
To show $\lambda^k\to\lambda^*$, it thus suffices to show $\vb^*=0$.

By \eqref{b*}, $\{\lambda^k\}$ is bounded and must have a convergent subsequence $\lambda^{k_j}\to \bar{\lambda}$.
Recall the optimality conditions for the $\bfx_i$-subproblems \eqref{Jacobian update}:
\beq
\label{xiopt0}
A_i^\top\left(\lambda^k-\rho(A_i\bfx^{k+1}_i+\sum_{j\neq i}A_j\bfx_j^k-c)\right)\in \partial f_i(\bfx_i^{k+1}).
\eeq
By Theorem 24.4 of \cite{rockafellar1997convex}, taking limit over the subsequence $\{k_j\}$ on both sides of
\eqref{xiopt0} yields:
\beq
A_i^\top\bar{\lambda}\in \partial f_i(\bfx_i^*),~\forall~i.
\eeq
Therefore, $(\bfx^*,\bar{\lambda})$ satisfies the KKT conditions of the problem \eqref{multivariateOptimization}.
Since $(\bfx^*,\lambda^*)$ is any KKT point, now we let $\lambda^*=\bar{\lambda}$. By \eqref{b*} and $\|\lambda^{k_j}-\lambda^*\|^2\to 0$, we must have $\vb^* = 0$, thereby completing the proof.
\end{proof}

Under the similar near-orthogonality assumption on the matrices $A_i,~i=1,2,\ldots,N$, we have the following convergence result for Proximal Jacobian ADMM:
\begin{theorem}
Suppose $\|A_i^\top A_j\|\leq \delta$ for all $i\neq j$, and the parameters in Algorithm \ref{Parallel_N} satisfy the following condition:
for some $\alpha,\beta>0$,
\begin{equation} \label{cond3}
\left\{\begin{array}{l}
P_i\succ \rho(\frac{1}{\alpha}-1)A_i^\top A_i + \frac{\rho}{\beta}\delta(N-1)\vI\\
\lambda_{\min}(A_i^\top A_i) > \frac{2-\gamma+\beta}{2-\gamma-\alpha}\delta(N-1)
\end{array}\right.
\mbox{for } i=1,\ldots,N.
\end{equation}
Then Algorithm \ref{Parallel_N} converges to a solution to the problem
\eqref{multivariateOptimization}.
\end{theorem}
\begin{proof}
Let
$$H:=\begin{pmatrix}
A_1^\top A_1 &\quad &\quad\\
\quad &\ddots &\quad\\
\quad &\quad &A_N^\top A_N
\end{pmatrix}.$$
If $\|A_i^\top A_j\|\leq \delta$ for all $i\neq j$, then it is easy to show the following: for any $\bfx$ and
$\bfy$,
\begin{align}
\|A\bfx\|^2 &= \sum_{i=1}^{N}\|A_i\bfx_i\|^2+\sum_{i\neq j}\bfx_i^\top A_i^\top A_j \bfx_j
\geq \sum_{i=1}^{N}\|A_i\bfx_i\|^2-\delta\sum_{i\neq j}\|\bfx_i\|\|\bfx_j\| \nonumber \\
&\geq \sum_{i=1}^{N}\|A_i\bfx_i\|^2-\delta (N-1)\|\bfx\|^2 = \|\bfx\|^2_{[H-\delta(N-1)\vI]},
\end{align}
and
\begin{align}
2\bfx^\top A^\top A\bfy &= 2\sum_{i=1}^{N}\bfx_i^\top A_i^\top A_j\bfy_j+2\sum_{i\neq j}\bfx_i^\top A_i^\top
A_j \bfy_j
\geq 2\sum_{i=1}^{N}\bfx_i^\top A_i^\top A_j\bfy_j-2\delta\sum_{i\neq j}\|\bfx_i\|\|\bfy_j\| \nonumber \\
&\geq -\sum_{i=1}^{N}\alpha\|A_i\bfx_i\|^2-\beta\delta (N-1)\|\bfx\|^2-\sum_{i=1}^{N}\frac{1}{\alpha}
\|A_i\bfy_i\|^2-\frac{1}{\beta}\delta(N-1)\|\bfy\|^2 \nonumber\\
&= -\|\bfx\|^2_{[\alpha H+\beta\delta(N-1)\vI]}-\|\bfy\|^2_{[\frac{1}{\alpha}H+\frac{1}{\beta}
\delta(N-1)\vI]}, ~\forall \alpha,\beta>0,
\end{align}
Using the above inequalities, we have
\begin{align}
&\frac{2}{\gamma}(\lambda^k-\lambda^{k+1})^\top A(\bfx^k-\bfx^{k+1}) = 2\rho(\bfx^{k+1}-\bfx^*)A^\top A(
\bfx^k-\bfx^{k+1})\nonumber\\
\geq &-\rho\|\bfx^{k+1}-\bfx^*\|^2_{[\alpha H+\beta\delta(N-1)\vI]}
-\rho\|\bfx^{k}-\bfx^{k+1}\|^2_{[\frac{1}{\alpha}H+\frac{1}{\beta}\delta(N-1)\vI]},
\end{align}
and
\begin{equation}
\|\lambda^k-\lambda^{k+1}\|^2 = \gamma^2\rho^2\|A(\bfx^{k+1}-\bfx^*)\|^2
\geq \gamma^2\rho^2\|\bfx^{k+1}-\bfx^*\|^2_{[H-\delta(N-1)\vI]}.
\end{equation}
Therefore,
\begin{align}
h(\bfu^k,\bfu^{k+1}) \geq &\|\bfx^k-\bfx^{k+1}\|^2_{G_x}+(2-\gamma)\rho\|\bfx^{k+1}-\bfx^*\|^2_{[H-
\delta(N-1)\vI]}\nonumber\\
&-\rho\|\bfx^{k+1}-\bfx^*\|^2_{[\alpha H+\beta\delta(N-1)\vI]}
-\rho\|\bfx^{k}-\bfx^{k+1}\|^2_{[\frac{1}{\alpha}H+\frac{1}{\beta}\delta(N-1)\vI]}.
\end{align}
As long as the following holds:
\begin{equation}
\left\{\begin{array}{l}
G_x\succ\frac{\rho}{\alpha}H+\frac{\rho}{\beta}\delta(N-1)\vI,\\
(2-\gamma)\rho[H-\delta(N-1)\vI]\succ \rho[\alpha H+\beta\delta(N-1)\vI],
\end{array}\right.
\end{equation}
which is equivalent to the condition \eqref{cond3}, there must exist some $\eta>0$ such that \eqref{h_bound} and \eqref{contr2}
hold. Then the convergence of Algorithm \ref{Parallel_N} follows immediately from the standard analysis of contraction methods \cite{he1997class}.
\end{proof}

\section{On $o(1/k)$ Convergence Rate of ADMM} \label{sec:ADMM_rate}
The convergence of the standard two-block ADMM has been long established in the literature \cite{gabay1976dual,glowinski1975approximation}.
Its convergence rate has been actively studied; see \cite{lions1979splitting,he2012on,he2012non,Goldstein2012fast,deng2012linear,luo2012linear}
and the references therein. In the following, we briefly review the convergence analysis for ADMM ($N=2$) and then improve the $O(1/k)$ convergence rate established in \cite{he2012non} slightly to $o(1/k)$ by using the same technique as in Subsection 2.2.

As suggested in \cite{he2012non}, the quantity $\|\vw^k-\vw^{k+1}\|_H^2$ can be used to measure the optimality of the iterations of ADMM , where $$
\vw:=\begin{pmatrix}
\bfx_2\\\lambda
\end{pmatrix},
~H:=\begin{pmatrix}
\rho A_2^\top A_2 &\quad\\
\quad &\frac{1}{\rho}\vI
\end{pmatrix},
$$
and $\vI$ is the identity matrix of size $m\times m$. Note that $\bfx_1$ is not part of $\vw$ because $\bfx_1$ can be regarded as an intermediate variable in the iterations of ADMM, whereas $(\bfx_2,\lambda)$ are the essential variables \cite{boyd2010distributed}. In fact, if $\|\vw^k-\vw^{k+1}\|_H^2=0$ then $\vw^{k+1}$ is optimal. The reasons are as follows.
Recall the subproblems of ADMM:
\begin{align}
\bfx_1^{k+1} &= \argmin_{\bfx_1}~f_1(\bfx_1)+\frac{\rho}{2}
\|A_1\bfx_1+A_2\bfx_2^k-\lambda^k/\rho\|^2,\\
\bfx_2^{k+1} &= \argmin_{\bfx_2}~f_2(\bfx_2)+\frac{\rho}{2}\|A_1\bfx^{k+1}_1+A_2
\bfx_2-\lambda^k/\rho\|^2.
\end{align}
By the formula for $\lambda^{k+1}$, their optimality conditions can be written as:
\begin{align}
\label{x1o}
A_1^\top\lambda^{k+1}-\rho A_1^\top A_2(\bfx_2^k-\bfx_2^{k+1})&\in \partial
f(\bfx_1^{k+1}),\\
\label{x2o}
A_2^\top\lambda^{k+1}&\in \partial f_2(\bfx_2^{k+1}).
\end{align}
In comparison with the KKT conditions \eqref{xsopt2} and \eqref{lsopt2}, we can see that $\bfu^{k+1}=\left(\bfx^{k+1}_1,\bfx^{k+1}_2,\lambda^{k+1}\right)$ is a solution of \eqref{multivariateOptimization} if and only if the following holds:
\begin{align}
\label{rp}
&\vr^{k+1}_p:=A_1\bfx_1^{k+1}+A_2\bfx_2^{k+1}-c = 0 \quad\mbox{(primal feasibility)},\\
\label{rd}
&\vr^{k+1}_d:=\rho A_1^\top A_2(\bfx_2^k-\bfx_2^{k+1}) = 0
\quad\quad\mbox{(dual feasibility)}.
\end{align}
By the update formula for $\lambda^{k+1}$, we can write $\vr_p$ equivalently as
\beq \label{admm3}
\vr^{k+1}_p = \frac{1}{\rho}(\lambda^k-\lambda^{k+1}).
\eeq
Clearly, if $\|\vw^k-\vw^{k+1}\|_H^2=0$ then the optimality conditions \eqref{rp} and \eqref{rd} are satisfied, so $\vw^{k+1}$ is a solution.
On the other hand, if $\|\vw^k-\vw^{k+1}\|_H^2$ is large, then $\vw^{k+1}$ is likely to be far away from being a solution.
Therefore, the quantity $\|\vw^k-\vw^{k+1}\|_H^2$ can be viewed as a measure of the distance between the iteration $\vw^{k+1}$ and the solution set.
Furthermore, based on the variational inequality \eqref{VI} and the variational characterization of the iterations of ADMM, it is reasonable to use the quadratic term $\|\vw^k-\vw^{k+1}\|_H^2$ rather than $\|\vw^k-\vw^{k+1}\|_H$ to measure the convergence rate of ADMM (see \cite{he2012non} for more details).

The work \cite{he2012non} proves that $\|\vw^k-\vw^{k+1}\|_H^2$ converges to zero at a rate of $O(1/k)$. The key steps of the proof are to establish the following properties:
\begin{itemize}
\item the sequence $\{\vw^k\}$ is contractive:\
\beq \label{xyopt}
\|\vw^k-\vw^*\|^2_H-\|\vw^{k+1}-\vw^*\|^2_H\geq\|\vw^{k}-\vw^{k+1}\|^2_H,
\eeq
\item the sequence $\|\vw^k-\vw^{k+1}\|_H^2$ is monotonically non-increasing:
\beq
\label{monotone2}
\|\vw^k-\vw^{k+1}\|^2_H\leq\|\vw^{k-1}-\vw^k\|^2_H.
\eeq
\end{itemize}
The contraction property \eqref{xyopt} has been long established and its proof dates back to \cite{gabay1976dual,glowinski1975approximation}.
Inspired by \cite{he2012non}, we provide a much shorter proof for \eqref{monotone2} than the one in \cite{he2012non}.

\begin{proof}[Proof of \eqref{monotone2}]
Let $\Dx_i^{k+1}=\bfx_i^k-\bfx_i^{k+1}$ and $\Dlambda^{k+1}=\lambda^k-\lambda^{k+1}$.
By Lemma \ref{lem:convex}, i.e., \eqref{convex}, the optimality condition \eqref{x1o} at the $k$-th and $(k+1)$-th iterations yields:
\beq
\langle \Dx_1^{k+1},~A_1^\top\Dlambda^{k+1}-\rho A_1^\top A_2(\Dx_2^k-\Dx_2^{k+1})\rangle \geq 0.
\eeq
Similarly for \eqref{x2o}, we obtain
\beq
\langle \Dx_2^{k+1},~A_2^\top\Dlambda^{k+1}\rangle \geq 0.
\eeq
Adding the above two inequalities together, we have
\beq \label{xyo2}
(A_1\Dx_1^{k+1}+A_2\Dx_2^{k+1})^\top\Dlambda^{k+1}-\rho (A_1\Dx_1^{k+1})^\top A_2(\Dx_2^k-\Dx_2^{k+1})\geq 0.
\eeq
Using the equality according to \eqref{admm3}:
\beq \label{admm3k}
A_1\Dx_1^{k+1}+A_2\Dx_2^{k+1} = \frac{1}{\rho}(\Dlambda^k-\Dlambda^{k+1}),
\eeq
\eqref{xyo2} becomes
\beq
\label{xyo3}
\frac{1}{\rho}(\Dlambda^k-\Dlambda^{k+1})^\top\Dlambda^{k+1}-(\Dlambda^k-\Dlambda^{k+1}-\rho A_2\Dx_2^{k+1})^
\top A_2(\Dx_2^k-\Dx_2^{k+1})\geq 0.
\eeq
After rearranging the terms, we get
\beq \label{xyo4}
\begin{split}
&(\sqrt{\rho}A_2\Dx_2^k+\frac{1}{\sqrt{\rho}}\Dlambda^k)^\top(\sqrt{\rho}A_2\Dx_2^{k+1}+\frac{1}{\sqrt{\rho}}
\Dlambda^{k+1})
-(A_2\Dx_2^k)^\top\Dlambda^k-(A_2\Dx_2^{k+1})^\top\Dlambda^{k+1}\\
&\geq \frac{1}{\rho}\|\Dlambda^{k+1}\|^2+\rho\|A_2\Dx_2^{k+1}\|^2=\|\vw^k-\vw^{k+1}\|_H^2.
\end{split}
\eeq

By the Cauchy-Schwarz inequality, we have
\beq
(a_1+b_1)^\top(a_2+b_2)\leq (\|a_1+b_1\|^2+\|a_2+b_2\|^2)/2,
\eeq
or equivalently,
\beq
(a_1+b_1)^\top(a_2+b_2)-a_1^\top b_1-a_2^\top b_2 \leq (\|a_1\|^2+\|b_1\|^2+\|a_2\|^2+\|b_2\|^2)/2.
\eeq
Applying the above inequality to the left-hand side of \eqref{xyo4}, we have
\beq
\begin{split}
\|\vw^k-\vw^{k+1}\|_H^2 &\leq \left(\rho\|A_2\Dx_2^{k}\|^2+\frac{1}{\rho}\|\Dlambda^{k}\|^2+\rho\|A_2\Dx_2^{k+1}
\|^2+\frac{1}{\rho}\|\Dlambda^{k+1}\|^2\right)/2\\
& = \left(\|\vw^{k-1}-\vw^{k}\|_H^2+\|\vw^k-\vw^{k+1}\|_H^2\right)/2,
\end{split}
\eeq
and thus \eqref{monotone2} follows immediately.
\end{proof}

We are now ready to improve the convergence rate from $O(1/k)$ to $o(1/k)$.
\begin{theorem}
\label{newrate}
The sequence $\{\vw^k\}$ generated by Algorithm~\ref{alg:admm} (for $N=2$) converges to a solution $\vw^*$ of
problem \eqref{multivariateOptimization} in the $H$-norm, i.e., $\|\vw^k-\vw^*\|_H^2\to0$, and
$\|\vw^k-\vw^{k+1}\|^2_H = o(1/k)$. Therefore,
\beq \label{o-rate}
\|A_1\bfx_1^k- A_1\bfx_1^{k+1}\|^2 + \|A_2\bfx_2^k - A_2\bfx_2^{k+1}\|^2 + \|\lambda^k- \lambda^{k+1}\|^2 =
o(1/k),
\eeq
for $k\to \infty$.
\end{theorem}
\begin{proof}
Using the contractive property of the sequence $\{\vw^k\}$ \eqref{xyopt} along with the optimality conditions,
the convergence of $\|\vw^k-\vw^*\|_H^2\to0$ follows from the standard analysis for contraction methods \cite{he1997class}.

By \eqref{xyopt}, we have
\beq
\sum_{k=1}^{n} \|\vw^k-\vw^{k+1}\|^2_H \leq \|\vw^{1}-\vw^*\|^2_H - \|\vw^{n+1}-\vw^*\|^2_H,~\forall n.
\eeq
Therefore, $\sum_{k=1}^{\infty} \|\vw^k-\vw^{k+1}\|^2_H<\infty$. By \eqref{monotone2},
$\|\vw^k-\vw^{k+1}\|^2_H$ is monotonically non-increasing and nonnegative. So Lemma \ref{lem:o} indicates that
$\|\vw^k-\vw^{k+1}\|^2_H = o(1/k)$, which further implies that $\|A_2\bfx_2^k - A_2\bfx_2^{k+1}\|^2 = o(1/k)$
and $\|\lambda^k- \lambda^{k+1}\|^2 = o(1/k)$. By \eqref{admm3k}, we also have $\|A_1\bfx_1^k-
A_1\bfx_1^{k+1}\|^2= o(1/k)$. Thus \eqref{o-rate} follows immediately.
\end{proof}
\begin{remark}
The proof technique based on Lemma \ref{lem:o} can be applied to improve some other existing convergence rates of $O(1/k)$ (e.g., \cite{he2013full,corman2013generalized}) to $o(1/k)$ as well.
\end{remark}

\section{Conclusion}
Due to the dramatically increasing demand for dealing with big data, parallel
and distributed computational methods are highly desirable.
ADMM, as a versatile algorithmic tool, has proven to be very effective at solving
many large-scale problems and well suited for distributed computing.
Yet, its parallelization still needs further investigation and improvement.
This paper proposes a simpler parallel and distributed ADMM for solving
problems with separable structures.
The algorithm framework introduces more flexibility for computing the
subproblems due to the use of proximal terms $\|\bfx_i-\bfx_i^k\|_{P_i}^2$ with
wisely chosen $P_i$. Its theoretical properties such as global convergence and
an $o(1/k)$ rate are established.
Our numerical results demonstrate the efficiency of the proposed method in comparison
with several existing parallel algorithms. The code will be available online for
further studies. In addition, we provide a simple sufficient condition to ensure the
convergence of Jacobian ADMM and demonstrate a simple technique to
improve the convergence rate of ADMM from $O(1/k)$ to $o(1/k)$.

\subsection*{Acknowledgements}
Wei Deng is supported by NSF grant ECCS-1028790.
Ming-Jun Lai is partially supported by a Simon collaboration grant for 2013--2018.
Wotao Yin is partially supported by NSF grants DMS-0748839 and DMS-1317602, and ARO/ARL MURI grant FA9550-10-1-0567.

\bibliographystyle{siam}
\bibliography{ADM,ADM_1}

\end{document}